\documentclass[10pt, reqno]{amsart}
\usepackage{amssymb}
\usepackage{hyperref}
\usepackage{cite}

\numberwithin{equation}{section}

\newcommand{\qtq}[1]{\quad\text{#1}\quad}

\newcommand{\R}{\mathbb{R}}

\newcommand{\eps}{\varepsilon}

\newtheorem{theorem}{Theorem}[section]

\newtheorem{lemma}[theorem]{Lemma}

\newtheorem{proposition}[theorem]{Proposition}

\theoremstyle{definition}
\newtheorem{definition}[theorem]{Definition}
\newtheorem{remark}[theorem]{Remark}

\theoremstyle{remark}

\begin{document}

\title[Cubic-quintic NLS]{Cubic-quintic NLS: scattering beyond\\the virial threshold}

\author[R. Killip]{Rowan Killip}
\address{Department of Mathematics, UCLA}
\email{killip@math.ucla.edu}

\author[J. Murphy]{Jason Murphy}
\address{Department of Mathematics \& Statistics, Missouri S\&T}
\email{jason.murphy@mst.edu}

\author[M. Visan]{Monica Visan}
\address{Department of Mathematics, UCLA}
\email{visan@math.ucla.edu}
\maketitle

\begin{abstract}  We consider the nonlinear Schr\"odinger equation in three space dimensions with combined focusing cubic and defocusing quintic nonlinearity.   This problem was considered previously by Killip, Oh, Pocovnicu, and Visan, who proved scattering for the whole region of the mass/energy plane where the virial quantity is guaranteed to be positive.  In this paper we prove scattering in a larger region where the virial quantity is no longer guaranteed to be sign definite.\end{abstract}

\section{Introduction}
We consider the cubic-quintic NLS in three space dimensions:
\begin{equation}\label{nls}
\begin{cases}
i\partial_t u + \Delta u = -|u|^2 u + |u|^4 u, \\
u|_{t=0} = u_0\in H^1(\R^3),
\end{cases}
\end{equation}
which describes the evolution of a complex field $u$ under the Hamiltonian
\[
E(u)=\int_{\R^3} \tfrac12 |\nabla u(x)|^2 + \tfrac16 |u(x)|^6 - \tfrac14 |u(x)|^4\,dx.
\]
This evolution also conserves the mass, defined by
\[
M(u)=\int_{\R^3}|u(x)|^2\,dx.
\]

We work here in $H^1(\R^3)$ because this is precisely the class of initial data for which both the mass and energy are finite.  The initial value problem \eqref{nls} was shown to be globally well-posed for such data by Zhang in \cite{Zhang}, building on the paper \cite{CKSTT} that treated the purely quintic nonlinearity.

We are concerned here with the long-time behavior of solutions.  For small initial data, standard arguments demonstrate that solutions scatter both forward and backward in time.  This means that there are functions $u_\pm\in H^1$ so that 
\begin{equation}\label{scattering}
\lim_{t\to\pm\infty} \|u(t)-e^{it\Delta}u_{\pm}\|_{H^1} = 0.
\end{equation}
On the other hand, our equation admits solitary wave solutions which evidently do not scatter.

The natural question then arises of determining the sharp threshold at which scattering breaks down.  There is already a large body of work on determining such thresholds for a wide variety of models, building on the paradigm introduced in \cite{KM}; see
\cite{Akahori',AkahoriNawa,Dodson1,Dodson2,DHR,FXC,Guevara,HR,KV,Miao},  for example.  This approach has two main steps: First one shows that the threshold for scattering is witnessed by a minimal counterexample that (by virtue of its minimality) is almost-periodic (modulo symmetries).  One then uses the virial identity (or close analogue) to prove that such almost periodic solutions cannot exist for initial data below the soliton threshold.

The virial identity for \eqref{nls} takes the following form:
\begin{equation*}
\tfrac{d\ }{dt} \ \bigl\langle u(t),\, \tfrac{1}{4i}( x\cdot\nabla + \nabla \cdot x) u(t)\bigr\rangle = V(u(t)),
\end{equation*}
where
\begin{equation}\label{E:V}
V(u)=\int |\nabla u(x)|^2 + |u(x)|^6 - \tfrac34 |u(x)|^4\,dx.
\end{equation}

As part of their wide-ranging investigation of the problem, the authors of \cite{KOPV} adapted this strategy to the problem of determining scattering thresholds for  \eqref{nls}.  In addition to all the positive results and the new ideas that were used to obtain them, the paper \cite{KOPV} also made an important contribution by discovering the inadequacy of the existing methods for obtaining a definitive scattering threshold.  Our ambition in this paper is to take a first step beyond the limitations of the existing approach.  Indeed, building on their work, we will be able to expand the frontier of the scattering region obtained in \cite{KOPV} everywhere that such expansion is not manifestly forbidden by the existence of solitons.

Many of the results of \cite{KOPV} are most readily framed with reference to Figure~\ref{F:1}.  This mass-energy diagram is purely schematic; while precise numerics are presented in \cite{KOPV}, the salient features live at very different length scales and so cannot be represented intelligibly on a single graph.

\begin{figure}[h]
\noindent
\begin{center}
\fbox{
\setlength{\unitlength}{1mm}
\begin{picture}(95,45)(-7,-7)
\put(0,-5){\vector(0,1){37}}\put(-5,30){$E$} 
\put(-5,0){\vector(1,0){87}}\put(82,-3){$M$} 
\put(30,0){\line(0,-1){2}}\put(30,-5){\hbox to 0mm{\hss$m_0$\hss}}
\put(40.5,0){\line(0,-1){2}}\put(40.5,-5){\hbox to 0mm{\hss$m_1$\hss}}
\put(65,0){\line(0,-1){2}}\put(65,-5){\hbox to 0mm{\hss$m_2$\hss}}
%
%
\linethickness{0.4mm}
\qbezier(36,18)(60,7)(65,0)
\qbezier(36,18)(50,12)(80,8)
%
%
\linethickness{0.05mm}
\qbezier(30,21)(30,25)(45,27)
\qbezier(30,21)(30,16.5)(55,13.5)
%
%
\multiput(30,21)(0,1){14}{\circle*{0.1}}
%
%
\qbezier(1,31)(02,32)(03,33)
\qbezier(0,27)(01,28)(06,33)\qbezier(0,24)(01,25)(09,33)
\qbezier(0,21)(03,24)(12,33)\qbezier(0,18)(04,22)(15,33)
\qbezier(0,15)(06,21)(18,33)\qbezier(0,12)(07,19)(21,33)
\qbezier(0,09)(09,18)(24,33)\qbezier(0,06)(10,16)(27,33)
\qbezier(0,03)(12,15)(30,33)
\qbezier(0,00)(15,15)(30,30)\qbezier(03,0)(16,13)(30,27)\qbezier(06,0)(18,12)(30,24)
\qbezier(09,0)(13,4)(15.6,6.6)\qbezier(18.7,9.7)(25,16)(30,21)\put(15.7,7){$\mathcal R$}
%
%
\qbezier(12,0)(15, 3)(31.1,19.1)
\qbezier(15,0)(18, 3)(33.0,18.0)
\qbezier(18,0)(21, 3)(35.2,17.2)
\qbezier(21,0)(24, 3)(37.5,16.5)
\qbezier(24,0)(27, 3)(39.9,15.9)
%
%
\qbezier(27,0)(30, 3)(42.1,15.1)
\qbezier(30,0)(33, 3)(44.1,14.1)
\qbezier(33,0)(36, 3)(46.1,13.1)
\qbezier(36,0)(39, 3)(48.1,12.1)
\qbezier(39,0)(42, 3)(50.0,11.0)
\qbezier(42,0)(45, 3)(52.0,10.0)
\qbezier(45,0)(48, 3)(53.9, 8.9)
\qbezier(48,0)(51, 3)(55.7, 7.7)
\qbezier(51,0)(54, 3)(57.5, 6.5)
\qbezier(54,0)(57, 3)(59.3, 5.3)
\qbezier(57,0)(60, 3)(61.0, 4.0)
\qbezier(60,0)(63, 3)(62.6, 2.6)
\qbezier(63,0)(64, 1)(64.1, 1.1)
\end{picture}
}
\end{center}
\caption{Schematic depiction of the mass-energy plane showing soliton solutions (heavy curve), non-soliton virial obstruction (light curve), and the scattering region $\mathcal R$ from \cite{KOPV}.}
\label{F:1}
\end{figure}
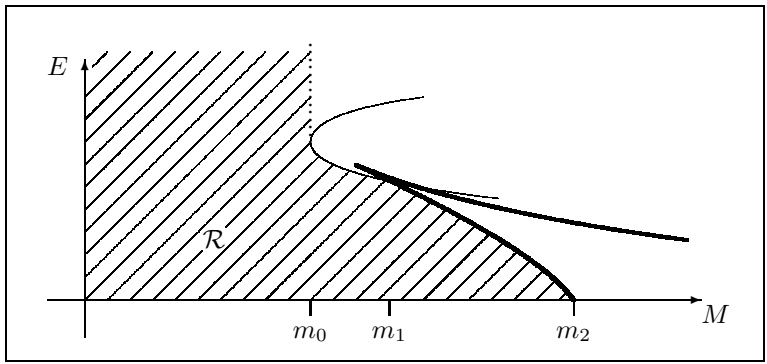

The heavy line in Figure~\ref{F:1} traces the mass-energy curve of ground-state soliton solutions.  By ground-state soliton, we mean an $H^1$-solution $P_\omega$ to
$$
-\Delta P_\omega + |P_\omega|^4P_\omega - |P|^2P_\omega + \omega P_\omega=0
$$
that is radially symmetric decreasing.  Such solutions exist if and only if $0<\omega<\tfrac3{16}$ and are unique for each such $\omega$; see \cite{KOPV}.

This mass-energy curve $\omega\mapsto(M(P_\omega),E(P_\omega))$ continues beyond the edges of Figure~\ref{F:1}. The convex branch is asymptotic to the $M$ axis, while the concave branch continues toward infinite mass and infinitely negative energy.  Note that we have only plotted the positive energy part of the mass-energy plane, since it is readily shown that no scattering solution can have negative energy.  We write $m_2$ for the point were the curve of solitons crosses the mass axis.  It is shown in \cite{KOPV} that for masses less than $m_2$, all solutions (other than zero) have positive energy.  While individual solutions with mass larger than $m_2$ may scatter, scattering cannot be guaranteed for any mass-energy pair in this regime. 

The mass $m_0$ is defined as the least mass at which it is possible to have $V(u)=0$ without $u\equiv 0$.  The first big surprise is that $m_0$ is strictly smaller than the smallest mass of any soliton (which coincides with the cusp in the heavy curve).  This was proved rigorously in \cite{KOPV}.  It is also shown there that $27 m_0^2= 16 m_2^2$.

This leads us to the natural question: What are these minimal objects that achieve zero virial?  The answer, discovered in \cite{KOPV}, is that they are exotic rescalings of soliton solutions in the sense that the rescaling preserves neither mass nor energy.  (Due to the combined nonlinearity, equation \eqref{nls} has no scaling symmetry.)  These rescaled solitons are defined by the equation
\begin{equation}\label{D:Romega}
R_\omega(x):=\sqrt{\tfrac{1+\beta(\omega)}{4\beta(\omega)}}\,P_\omega\!\Bigl(\tfrac{3(1+\beta(\omega))}{4\sqrt{3\beta(\omega)}}x\Bigr),
	\qtq{where} \beta(\omega):=\tfrac{\|P_\omega\|_{L^6}^6}{\|\nabla P_\omega\|_{L^2}^2}.
\end{equation}
The mass-energy curve of these rescaled solitons is shown as the light weight curve in Figure~\ref{F:1}.  

Consider now the open region $\mathcal{R}$ shown in Figure~\ref{F:1}.  The boundary of $\mathcal R$, which is not included in $\mathcal R$, is comprised partly of the mass-energy curve of solitons (heavy curve), partly by the curve of rescaled solitons (light curve), and partly by the mass $m_0$ (dotted line). 

In \cite{KOPV}, it is shown that $V(u)>0$ for every state $u$ whose mass and energy lies in the region $\mathcal R$.  In fact, the region $\mathcal R$ is the  maximal region with this property in the following sense:  Defining the function
\begin{equation}\label{EVdefn}
E^V_{\min}(m) =\inf\{E(u):u\in H^1(\R^3),\ M(u)=m,\ \text{and }V(u)=0\},
\end{equation}
we have $\mathcal R = \{(m,e) : e<E_V(m)\}$.  Note that $E^V_{\min}(m)=\infty$ for $m<m_0$.

In this paper, we prove that scattering holds in a still larger region $\mathcal B$ of the mass-energy plane:

\begin{theorem}\label{T} There is an open region $\mathcal B\subseteq \R^2$ so that every solution $u$ to \eqref{nls} with $\bigl(M(u),E(u)\bigr)\in \mathcal B$ belongs to $L^{10}(\R\times\R^3)$ and so scatters (in both time directions) in the sense of \eqref{scattering}.  The region $\mathcal B$ is larger than $\mathcal R$:\\
(a) There is an $m>m_0$ so that $\mathcal B\supseteq(0,m)\times(0,\infty)$; and\\
(b) If $(m,e)\in\partial \mathcal R$ is not achieved by a soliton, then $(m,e)\in\mathcal B$.
\end{theorem}

With reference to Figure~\ref{F:1}, we see from (a) that the whole dotted line on the boundary may be moved to the right (remaining vertical!).  Part (b) then completes our promise to expand the region of scattering in every place where this is not rigorously forbidden by the existence of solitons (which evidently do not scatter).  That is, we extend the scattering region across the entire portion of $\partial \mathcal R$ comprised by the light-weight curve.

In Figure~\ref{F:1} we see that for $m\geq m_0$, the left portion of $\partial \mathcal{R}$ is delimited by rescaled solitons and the right portion by solitons.  The mass at which this transition takes place is marked $m_1$.  The validity of this description was shown in \cite{KOPV} via numerics.  The purely analytical arguments given there show only that the left-most portion is made up of rescaled solitons and that the right-most part is made up solely of solitons; the possibility that roles are exchanged multiple times is not excluded.

The scattering claim in Theorem~\ref{T} follows from a more quantitative assertion that we will prove, namely, that the spacetime bound
\begin{equation}\label{STB}
\|u\|_{L_{t,x}^{10}(\R\times\R^3)} \leq C\bigl(M(u),E(u)\bigr)
\end{equation}
holds for every solution $u$ with $(M(u),E(u))\in \mathcal B$.  Here $C$ is some unspecified function.  The same type of bound was proven in \cite{KOPV} for $(m,e)\in\mathcal R$ and indeed for other equations in many papers preceding it.  One novelty of \cite{KOPV} is the appearance of two parameters (both mass and energy) in RHS\eqref{STB}.  It is worth emphasizing that due to the concavity of the soliton portion of $\partial\mathcal R$, the arguments in \cite{KOPV} cannot be based on a free energy of the form $E+\lambda M$ as had proven successful for a number of similar problems.  The need to induct on two parameters leads us to consider the following partial order on mass-energy pairs:

\begin{definition}\label{D:SW} We write $(m,e)\preccurlyeq (m_0,e_0)$ to indicate that $m\leq m_0$ and $e\leq e_0$.  That is, $(m,e)$ is southwest of $(m_0,e_0)$ in the mass-energy plane.
\end{definition}

With these preliminaries set, we can establish the following proposition.  As we will explain, this is essentially a recapitulation of the concentration compactness analysis in \cite{KOPV}. 

\begin{proposition}\label{P:cc}
Fix $0<m<m_2$ and $e>0$.  Then exactly one of the two following possibilities hold:\\
{\upshape (i)} There exists $\eps>0$ so that
\begin{equation}\label{STB'}
\|u\|_{L_{t,x}^{10}(\R\times\R^3)} \leq \eps^{-1} \qtq{whenever}  (M(u),E(u))\preccurlyeq (m+\eps,e+\eps).
\end{equation}
{\upshape (ii)} There is a global solution $u\in C_t(\R;H^1)$ to \eqref{nls} with $(M(u),E(u))\preccurlyeq (m,e)$.  Moreover, for this solution there
is a function $c:\R\to \R^3$ so that 
\begin{equation}\label{AP}
\{ u(t,x-c(t)) : t\in \R\} \text{ is precompact in $H^1(\R^3)$ and } c(t)=o(t) \text{ as }t\to \pm\infty.
\end{equation}
\end{proposition}

\begin{proof}
If (i) holds at the original pair $(m,e)$, then the matter is settled: Scattering holds and so \eqref{AP} cannot.

Henceforth, we assume that (i) does not hold for any $\eps>0$ and must prove (ii).  For this purpose, we may safely replace $(m,e)$ with any of its $\preccurlyeq$-predecessors.  We choose $(m,e)$ to be a $\preccurlyeq$-minimal pair for which (i) fails.

As (i) fails at $(m,e)$, there must be a sequence of solutions $u_n(t)$ to \eqref{nls} with
$$
M(u_n)\to m,\quad E(u_n)\to e,\qtq{and}  \bigl\|u_n\bigr\|_{L_{t,x}^{10}(\R\times\R^3)} \to \infty.
$$
Applying \cite[Proposition~9.1]{KOPV} exactly as in the proof of \cite[Theorem~9.6]{KOPV} we obtain the existence of a solution $u$ satisfying the precompactness claim in \eqref{AP}.  This compactness relies on the minimality of $(m,e)$.  That the spatial center function $c(t)$ must be $o(t)$ also follows from this minimality; see \cite[Proposition~10.2]{KOPV}.
\end{proof}

We are now ready to discuss how we are to prove Theorem~\ref{T}.  The argument has two main phases.  In phase one, we prove that a point in $\partial\mathcal R$ can only support an almost periodic solution (that is, one satisfying \eqref{AP}) if it supports a solitary wave solution.  It then follows from Proposition~\ref{P:cc} that scattering extends to a small neighborhood of any point $(m,e)\in\partial\mathcal R$ that is not achieved by a soliton.  The key observation driving the phase-one argument is that while individual functions with mass-energy belonging to $\partial \mathcal R$ may have vanishing virial, the only trajectories that maintain zero virial are the soliton solutions.  Actually, this is not quite enough:  In order to preclude almost periodic solutions we need a quantitative lower bound.  Evidently, this cannot hold pointwise in time; nevertheless, we will be able to prove that such a bound does hold for the integral over fixed-size time intervals.

While phase one provides a resulting scattering region that fulfills part (b) of Theorem~\ref{T}, there is no reason to believe it satisfies (a).  However, it does show that given $E_*$, there is an $m=m(E_*)>m_0$ so that scattering holds in the rectangle $(0,m)\times(0,E_*)$.  To complete phase two of the proof, we show that there is a suitable choice of $E_*$ that allows us to employ a different argument to prove that scattering extends to the entire strip $(0,m)\times(0,\infty)$.

\section*{Acknowledgments} R.~K. was supported by NSF grant DMS-1856755.  J.~M. was supported by a Simons Collaboration Grant.  M.~V. was supported by NSF grant DMS-1763074.

\section{Preliminaries}

In order to prove scattering, we must preclude solutions to \eqref{nls} with the property \eqref{AP}.  Like many of our predecessors, we will do this using a localized version of the virial identity.  The novel aspect of this paper is \emph{how} we falsify \eqref{LV-LB}.

\begin{proposition}[The localized virial argument]\label{P:LV} If $u$ is a solution to \eqref{nls} satisfying \eqref{AP}, then 
\begin{equation}\label{LV-LB}
\lim_{T\to\infty} \frac{1}{T}\int_0^T  V(u(t))\,dt = 0,
\end{equation}
where $V(u)$ is as in \eqref{E:V}. 
\end{proposition}

\begin{proof} This is the essence of the standard localized virial argument; details can be found in \cite[p. 542--543]{KOPV}.  The only difference is that there the argument is by contradiction since it is already known that for the solutions of interest, $V(u(t))$ is bounded away from zero (uniformly for $t\in\R$).
\end{proof}

In order to upgrade certain qualitative statements from \cite{KOPV} to quantitative ones, we employ the following from \cite{HmidiKeraani2005}:

\begin{proposition}[Profile decomposition for Gagliardo--Nirenberg]\label{P:PDGN} Let $\{u_n\}$ be a bounded sequence in $H^1(\R^3)$.  Passing to a subsequence if necessary, there exist $J^*\in\{0,1,2,\dots\}\cup\{\infty\}$, profiles $\phi^j\in H^1\backslash\{0\}$, and positions $\{x_n^j\}\subset\R^3$, so that
\begin{equation}\label{E:PDGN}
u_n(x) = \sum_{j=1}^J \phi^j(x-x_n^j)+w_n^J(x)
\end{equation}
for each finite $0\leq J\leq J^*$.  Moreover, the following hold:
\begin{align*}
& \lim_{J\to J^*}\limsup_{n\to\infty} \|w_n^J\|_{L^4} = 0, \\
& \sup_J \limsup_{n\to\infty} \ \Bigl| \|u_n\|_{\dot H^1}^2 - \sum_{j=1}^J \|\phi^j\|_{\dot H^1}^2 - \|w_n^J\|_{\dot H^1}^2 \Bigr| = 0, \\
&\sup_J \limsup_{n\to\infty} \ \Bigl| \| u_n\|_{L^q}^q  - \sum_{j=1}^J \|\phi^j\|_{L^q}^q - \|w_n^J\|_{L^q}^q \Bigr| = 0\quad\text{for} \quad q\in\{2, 4,6\}.
\end{align*}
\end{proposition}

Recall that $m_0$ denotes the least (non-zero) mass at which $V(u)=0$ is possible.  This quantity is determined through a variational problem analyzed in \cite{KOPV}, which will be important for analyzing the vertical portion of $\partial \mathcal R$.
 
\begin{proposition}[Interpolation inequality]\label{P31}  Every $u\in H^1(\R^3)$ satisfies
\begin{equation}\label{GNH}
\tfrac34\int |u|^4 \,dx  \leq \sqrt{\tfrac{M(u)}{m_0}}  \int |\nabla u|^2 + |u|^6\,dx . 
\end{equation}
Moreover, there is a non-empty finite set $\Omega\subset(0,\tfrac3{16})$ which characterizes optimizers thus:  $M(u)=m_0$ and equality holds in \eqref{GNH} if and only if $u(x)= e^{i\theta} R_\omega(x-c)$ for some $c\in\R^3$, $\theta\in\R$, and $\omega\in\Omega$.
\end{proposition}

\begin{remark}\label{R:argg}
The numerics in \cite{KOPV} show that $\Omega$ is actually comprised of a single point; this would also be a consequence of Conjecture~2.6 of that paper.  This is what is depicted in Figure~\ref{F:1}.  While the numerics are very stable and compelling, we do not have a rigorous proof that there is only one value of $\omega$ so that $M(R_\omega)=m_0$.  Thus, we will not make this assumption in what follows.
\end{remark}

\begin{proof}
Choosing $\alpha=1$ in \cite[Proposition~3.1]{KOPV} shows that the optimal constant in the inequality
\begin{equation}\label{GNH'}
\|u\|_{L^4}^4 \leq C \|u\|_{L^2} \|\nabla u\|_{L^2}^{\frac32} \|u\|_{L^6}^{\frac32}
\end{equation}
is realized.  Moreover, every optimizer takes the form $u(x)=e^{i\theta}\lambda R_\omega(\rho(x-c))$ with $\theta\in\R$, $\lambda>0$, $\rho>0$, $c\in\R^3$, and $\beta(\omega)=1$.  In \cite{KOPV}, this is written with $P_\omega$, rather than $R_\omega$, but these assertions are equivalent in view of \eqref{D:Romega}.

The value of $m_0$ was determined in \cite{KOPV} from the optimal constant $C$ in \eqref{GNH'} via
\begin{equation}\label{C and m_0}
m_0 = \sqrt{3} \bigl[\tfrac{16}{9C}\bigr]^2
\end{equation}
for details, see \cite[Lemma~3.3 and Theorem~5.2]{KOPV}. 

From \cite[Theorem~2.2]{KOPV} we see that $\beta$ is an analytic function of $\omega\in(0,\tfrac3{16})$ and that $\beta(\omega)\to 0$ as $\omega\to0$ and $\beta(\omega)\to \infty$ as $\omega\to\tfrac3{16}$.  Thus choosing $\Omega$ to denote the set of all $\omega$ that arise as optimizers, we are guaranteed that $\Omega$ is a finite set.  Note that, at this time, we cannot claim that every solution to $\beta(\omega)=1$ corresponds to an optimizer.

The inequality \eqref{GNH} follows from \label{C and m_0} and \eqref{GNH'} via Young's inequality in the form
$$
3^{1/4} a b \leq \tfrac34 \bigl[ a^{4/3} + b^4 \bigr] \qtq{with equality iff} a^{4/3}=3 b^4.
$$
This requirement for equality places one constraint on the scaling parameters $\rho$ and $\lambda$.  Combining this with the requirement that $M(u)=m_0$ then guarantees that $\lambda=1$ and $\rho=1$ as stated in the proposition.  Details of these computations can also be found in \cite[Lemma~5.5]{KOPV}.
\end{proof}

\section{Proof of Theorem~\ref{T}}

Let us write $\partial\mathcal R_s$ to represent the portion of $\partial\mathcal R$ represented by solitary waves:
$$
\partial\mathcal R_s := \bigl\{ (m,e)\in\partial\mathcal R : (m,e)=\bigl(M(P_\omega),E(P_\omega)\bigr)\text{ for some }\omega\in(0,\tfrac3{16})\bigr\}.
$$
The remainder of the boundary will be denoted $\partial\mathcal R_r$.  Our first goal in this section is to prove Proposition~\ref{P:phase 1}, which shows that scattering holds in a neighborhood of any mass-energy pair $(m,e) \in \partial\mathcal R_r$.

Before we can begin on the novel portion of the analysis, we need one more lemma recapitulating material developed in \cite{KOPV}:

\begin{lemma}\label{L:PS00} Fix $(m,e)\in\partial\mathcal R_r$ and let $u\in H^1$ satisfy $(M(u),E(u))\preccurlyeq (m,e)$.  Then $V(u)\geq 0$.  Moreover, if
$V(u)=0$, then either $u\equiv 0$ or $u(x) = e^{i\theta} R_\omega(x-c)$ for some $c\in\R^3$ and $\theta\in[0,2\pi)$.  The value of $\omega$ is uniquely determined by $E(u)$.
\end{lemma}

\begin{proof}
When $m=m_0$, both $V(u)\geq 0$ and the identification of cases of equality follow from Proposition~\ref{P31}.  When $m>m_0$, we rely instead on \cite[Theorem~5.6]{KOPV}, which identifies optimizers for \eqref{EVdefn} at fixed mass, and \cite[Theorem~5.2]{KOPV}, which shows strict monotonicity of $m\mapsto E^V_{\min}(m)$.

Finally, the monotonicity \cite[Equation (5.19)]{KOPV} guarantees that $E(R_\omega)$ uniquely determines $\omega$.
\end{proof}

For our purposes, it is not enough merely to understand which functions achieve zero virial.  Rather, we need to understand minimizing sequences:

\begin{lemma}\label{L:PS3} Fix $(m,e)\in\partial\mathcal R_r$ and let $u_n$ be a sequence in $H^1(\R^3)$ with 
$M(u_n)=m$ and $E(u_n)=e$.  If $V(u_n) \to 0$ then there are $x_n\in\R^3$ and $\theta_n\in[0,2\pi)$ so that
\begin{equation}\label{E:to R}
e^{i\theta_n} u_n(x - x_n) \longrightarrow R_\omega(x)\quad\text{in $H^1(\R^3)$,}
\end{equation}
where $\omega$ is uniquely determined by $E(R_\omega)=e$.
\end{lemma}

\begin{proof}
We apply the profile decomposition of Proposition~\ref{P:PDGN} to the sequence $u_n$.  The remaining analysis will be confined to the resulting subsequence where \eqref{E:PDGN} holds.  This suffices, for if the claim were false, there would be a sequence $u_n$ without any subsequence satisfying \eqref{E:to R}.

From the conclusions of Proposition~\ref{P:PDGN}, we know that the profiles satisfy
\begin{equation}\label{E:to Ra}
 \limsup_{J\to J^*} \ \limsup_{n\to\infty}\  \biggl[M(w_n^J) + \sum_{j=1}^J M(\phi^j)\biggr] = m, \qquad \sum_{j=1}^{J^*}  E(\phi^j ) \leq e, 
\end{equation}
and moreover,
\begin{equation}\label{E:to Rb}
0 =  \limsup_{J\to J^*} \ \limsup_{n\to\infty}\ \biggl[ \| \nabla w_n^J\|_{L^2}^2 + \| w_n^J \|_{L^6}^6 + \sum_{j=1}^J V(\phi_j) \biggr].
\end{equation}

From these relations and Lemma~\ref{L:PS00}, we see that any $\phi^j$ (which are always non-zero) must agree with some (a priori $j$-dependent) rescaled soliton $R_\omega$ up to a translation and phase rotation.  Modifying the symmetry parameters in \eqref{E:PDGN}, if necessary, we may assume that the profile is exactly $R_\omega$.  Lemma~\ref{L:PS00} and \eqref{E:to Rb} also show that $\nabla w_n^J \to 0$ in $L^2$.

In view of the mass and energy constraints in \eqref{E:to Ra} and the monotonicity of $m\mapsto E_{\min}^V(m)$ for $m\geq m_0$, we are left with two possibilities: Either there is exactly one profile $\phi^1=R_\omega$ and $w_n\to 0$ in $L^2$-sense, or $u_n = w_n$ for all $n$.  The former case yields \eqref{E:to R}.  The latter case is incompatible with $V(u_n)\to 0$, because $w_n\to 0$ in $L^4$ and so $\liminf_{n\to\infty} V(w_n) \geq E(w_n) \equiv e$.
\end{proof}

Recalling that our ambition is to contradict \eqref{LV-LB} (and thereby preclude solutions obeying \eqref{AP}), we can now see how Lemma~\ref{L:PS3} helps:  It shows that the virial will remain bounded away from zero unless the trajectory passes near $R_\omega$.  To handle the remaining case, we introduce the notation
\begin{equation}\label{rho defn}
\rho_\omega(u):=\inf\Big\{ \bigl\|u(x)-e^{i\theta}R_{\omega}(x-c)\bigr\|_{H^1} : c\in\R^3,\ \theta\in\R\Bigr\}
\end{equation}
and prove the following:

\begin{lemma}\label{L:delta'} Fix $\omega\in(0,\frac3{16})$.  Then there exists $\delta>0$ so that every solution $u$ to \eqref{nls} with $\rho_\omega(u(0))\leq \delta$ satisfies
\begin{equation}\label{E:delta'}
\int_{-1}^1 V(u(t))\,dt \geq \delta .
\end{equation}
\end{lemma} 

\begin{proof}
The crux of the proof is the following: If $v$ is the solution to \eqref{nls} with initial data $v(0)=R_\omega$, then
\begin{equation}\label{E:delta''}
\int_{-1}^1 V(v(t))\,dt  >0 .
\end{equation}
We will prove this shortly.  As $V$ is continuous on $H^1$, it follows from this assertion and local well-posedness that there is a $\delta>0$ so that \eqref{E:delta'} holds for every solution with $\|u(0)-R_\omega\|_{H^1}\leq \delta$.  This extends to solutions with $\rho_\omega(u)\leq\delta$ due to the gauge and translation symmetries of \eqref{nls}.

We now turn our attention to \eqref{E:delta''}, arguing by contradiction.  From Lemma~\ref{L:PS3}, we know $V(v(t))\geq 0$ and thus the failure of \eqref{E:delta''} ensures that $V(v(t))\equiv 0$ for $t\in[-1,1]$.  But this in turn guarantees
$$
v(t,x) = e^{i\theta(t)} R_\omega(x-x(t))
$$
for some (continuous) functions $\theta(t)$ and $x(t)$.  In view of the uniqueness of $H^1$-solutions to \eqref{nls}, the rotational symmetry of \eqref{nls} and the initial data $v(0)$ guarantee that $x(t)\equiv 0$.  Analogously, combining gauge symmetry and time-translation invariance guarantee that $\theta(t+s)=\theta(t)+\theta(s)$ whenever $s,t,s+t\in[-1,1]$. This in turn shows that $\theta(t)=\xi t$ for some $\xi\in\R$ and consequently that $v(t,x)=e^{i\xi t} R_\omega(x)$ for $t\in[-1,1]$.  By gauge invariance this shows $v$ to be a solitary wave solution for all $t\in \R$, which is inconsistent with the assumption $(M(R_\omega),E(R_\omega))\in \partial\mathcal R_r$.
\end{proof}

With Lemmas~\ref{L:PS3} and~\ref{L:delta'} in place, we are now ready to demonstrate the existence of the new, larger scattering region that satisfies property (b) listed in Theorem~\ref{T}, completing phase one of our argument.  The region will be further enlarged in Proposition~\ref{P:phase 2} so that property (a) also holds.

\begin{proposition}\label{P:phase 1}
Given any pair $(m,e)\in\partial\mathcal R_r$,  there exist $\eps>0$ and $C<\infty$ so that
for any solution $u$ to \eqref{nls},
$$
(M(u),E(u))\preccurlyeq (m+\eps,e+\eps) \qtq{implies} \|u\|_{L_{t,x}^{10}(\R\times\R^3)} \leq C.
$$
\end{proposition}

\begin{proof}
In view of Proposition~\ref{P:cc}, it suffices to show that there are no solutions $u$ satisfying \eqref{AP} with $(M(u),E(u))\preccurlyeq (m,e)$.  This in turn will be effected by falsifying \eqref{LV-LB}.

Suppose that $u$ is a solution satisfying \eqref{AP} and $(M(u),E(u))\preccurlyeq (m,e)$.  By the main theorem from \cite{KOPV}, we must have $(M(u),E(u))\in \partial\mathcal R_r$.  We first consider the case that this mass-energy pair does \emph{not} coincide with that of any $R_\omega$.  This assumption combined with Lemma~\ref{L:PS3} shows that $V(u(t_n))\to 0$ is impossible for any sequence of times $t_n$.  Thus it follows that
\begin{align}\label{V big}
V(u(t))\gtrsim 1 \text{ uniformly for $t\in\R$},
\end{align}
which clearly contradicts \eqref{LV-LB}. 

Suppose now that the mass and energy of $u$ do coincide with those of a (necessarily unique) rescaled solitary wave $R_\omega$.  Let $\delta>0$ be that given by Lemma~\ref{L:delta'} for this value of $\omega$.  By Lemma~\ref{L:PS3}, there exists $\eps>0$ so that
\begin{equation}\label{E:ps33}
V(u(t)) \geq \eps \qtq{whenever} \rho_\omega(u(t)) \geq \delta .
\end{equation}
Now given $T>0$ large, let $K_T=\{t\in[1,T-1] : \rho_\omega(u(t))\geq \delta\}$.  The (trivial) one-dimensional Besicovich covering lemma shows that one can find a finite collection of times $t_i\in K_T$ so that
$$
K_T \subseteq K_T^* := \bigcup \, [t_i-1,t_i+1]
$$
and no $t\in\R$ belongs to more than two intervals $[t_i-1,t_i+1]$.   Applying \eqref{E:delta'} to each such interval, \eqref{E:ps33} on $[1,T-1] \setminus K_T^*$, and neglecting the positive contribution of any points not included, we find
$$
\int_0^T  V(u(t))\,dt \geq \tfrac12 \delta |K_T^*|  + \eps \bigl| [1,T-1] \setminus K_T^* \bigr| .
$$
As neither $\delta$ nor $\eps$ depend on $T$, this conclusion is inconsistent with \eqref{LV-LB} and so completes the proof of the proposition in the case that the pair $(m,e)\in\partial \mathcal R_r$ coincides with that of some rescaled soliton.
\end{proof}

\begin{proposition}\label{P:phase 2}
There exists $m>m_0$ and a function $C:[0,\infty)\to \R$ so that any solution $u$ to \eqref{nls} with $M(u)\leq m$ satisfies
\begin{equation}\label{2stb}
 \|u\|_{L_{t,x}^{10}(\R\times\R^3)} \leq C(E(u)).
\end{equation}
\end{proposition}

\begin{proof}
Proposition~\ref{P:phase 1} together with a simple covering argument shows that for any choice of $E_*$, there is an $m >m_0$ so that \eqref{2stb} holds for all solutions with $M(u)\leq m$ and $E(u)\leq E_*$.

We choose $E_*=m_2$, which is guaranteed to be larger than the corresponding $m$ due to the existence of solitary waves.  It follows that if $M(u) \leq m$ and $E(u) \geq E_*$, then
$$
V(u) = 2E(u)  - \tfrac3{64} M(u) + \tfrac13 \int |u|^2\bigl[ |u|^2 - \tfrac38]^2\,dx \geq m_2 .
$$
It is evident from Proposition~\ref{P:LV} that no solution to \eqref{nls} can satisfy both these mass-energy constraints and \eqref{AP}.  On the other hand, $m$ was chosen so that no solutions can satisfy \eqref{AP} with $(M(u),E(u))\preccurlyeq (m,E_*)$.   This proves the proposition, since Proposition~\ref{P:cc} shows that the failure of \eqref{2stb} would produce exactly the type of solution that we have just precluded.
\end{proof}


\end{document}